\documentclass[12pt]{amsart}
\RequirePackage[sc]{mathpazo}
\counterwithin{equation}{section}
\usepackage{eucal}

\usepackage[centering,scale=0.75]{geometry}

\usepackage[all]{xy}
\usepackage{amssymb}
\usepackage{mathrsfs}
\usepackage{graphicx}
\usepackage{xcolor}
\usepackage{tikz}
\usepackage[colorlinks,plainpages,backref,
    linkcolor=blue!80!green,
    citecolor=green!60!blue,
    urlcolor=red!50!black]{hyperref}

\newtheorem{theorem}{Theorem}[section]

\newtheorem{lemma}[theorem]{Lemma}

\newtheorem{prop}[theorem]{Proposition}

\newtheorem{question}[theorem]{Question}

\theoremstyle{remark}
\newtheorem{remark}[theorem]{Remark}

\begin{document}

\title{Permutation representations on cohomology of toric varieties}

\author{Tao Gui}
\address{(Tao Gui) \newline \indent Institute for Theoretical Sciences, Westlake University, No.\ 600 Dunyu Road, Xihu District, Hangzhou, Zhejiang Province 310030, China}
\email{guitao18(at)mails(dot)ucas(dot)ac(dot)cn}

\author{Chushi Qin}
\address{(Chushi Qin)\newline \indent Institute for Interdisciplinary
Information Sciences, Tsinghua University, Beijing 100084, China}
\email{tcs24@mails.tsinghua.edu.cn}

\author{Kaizhe Shen}
\address{(Kaizhe Shen)\newline \indent Department of Mathematics, 1291 University of Oregon, 175 Prince Lucien Campbell Hall, Eugene, OR 97403-1205}
\email{kaizhe@uoregon.edu}

\author{Rui Xiong}
\address{(Rui Xiong)\newline \indent Department of Mathematics, University of Michigan, Ann Arbor, MI 48109}
\email{xiongrui@umich.edu}



\begin{abstract}
Let $G$ be a finite group acting properly by lattice automorphisms on a complete simplicial fan $\Sigma$. An open question due to Stanley asked whether the (ungraded) representation carried by the cohomology $H^*(X_{\Sigma})$ of the associated toric variety $X_{\Sigma}$ is isomorphic to a  permutation representation of $G$. We prove that Stanley's question has an affirmative answer for all smooth projective toric varieties without the properness assumption on the action. 
The proof is inspired by toric mirror symmetry.
\end{abstract}

\maketitle

\section{Introduction}

A fundamental theorem in toric geometry states that there is a correspondence between (normal) toric varieties and fans of strongly convex rational (with respect to some lattice) polyhedral cones~\cite{Fulton}. This correspondence gives a rich dictionary between the geometry of a toric variety and the combinatorics of its associated fan. For example, the toric variety is rationally smooth if and only if the associated fan is simplicial, the toric variety is smooth if and only if each cone of $\Sigma$ is generated by part of a $\mathbb{Z}$ basis of the lattice, and the toric variety is compact in the classical topology if and only if the associated fan is complete, see~\cite[Theorem 3.1.19]{CLS}. Furthermore, the toric variety is projective if and only if the associated fan is polytopal (that is, it is the normal fan of some rational convex polytope)~\cite[Theorem 6.2.1]{CLS}. 

For a complete simplicial toric variety $X_\Sigma$ associated with a complete simplicial fan $\Sigma$, the cohomology ring $H^{*}(X_\Sigma)$ of the toric variety has a particularly nice presentation using the Stanley--Reisner ring of the complete simplicial fan $\Sigma$. 

If a finite group $G$ acts on a rational fan $\Sigma$ preserving the defining lattice, then this action induces an action of $G$ on the associated toric variety $X_\Sigma$, and hence induces a graded representation of $G$ on its cohomology $H^{*}(X_\Sigma)$. For the case of a Weyl group $W$ acting on the $W$-permutohedron and the normal fan consisting of Weyl chambers, this representation has been extensively studied by Procesi~\cite{Procesi}, by Dolgachev--Lunts~\cite{DolgachevLunts}, by Stembridge~\cite{Stembridge}, and by Lehrer~\cite{Lehrer}. 
In~\cite[Theorem 1.4]{Stembridge}, Stembridge  derived a graded character formula for a finite group $G$ acting on $H^{*}(X_\Sigma)$, which is induced from the action of $G$ on the complete simplicial fan $\Sigma$, with the additional (and very strong) assumption that the action of $G$ on $\Sigma$ is \emph{proper}. By definition, an action of $G$ on $\Sigma$ is proper if, for every $g \in G$, whenever $g$ fixes a cone $\sigma$ setwise, it fixes $\sigma$ pointwise. That is, the $g$-action on
any $g$-stable cone is trivial for any $g\in G$.

More recently, the first-named author derived a general graded character formula (an equivariant version of the classical h-polynomial) for the action of any finite group $G$ on the cohomology $H^*(X_{\Sigma})$ of a complete simplicial toric variety $X_{\Sigma}$, which is induced from the action of $G$ on the complete simplicial fan $\Sigma$, without the properness assumption~\cite{Gui}.

In the 1990s, an open question due to Stanley (recorded in~\cite[Question 11.1]{Stembridge}) asked whether the representation of $G$ on the cohomology $H^{*}(X_{\Sigma})$ of a complete simplicial toric variety $X_{\Sigma}$, which is induced by a \emph{proper} action of $G$ on the complete simplicial fan $\Sigma$, is a \emph{permutation representation} of $G$. 

\begin{question}[Stanley] \label{ques-permu}
    If $\Sigma$ is a rational complete simplicial fan that carries a proper action of a finite group $G$, is the (ungraded) representation carried by $H^*(X_{\Sigma})$ isomorphic to a permutation representation of $G$~?
\end{question}

For the fans associated with Weyl chambers of Weyl groups, Stembridge~\cite{Stembridge} showed that the above question has an affirmative answer using case-by-case analysis for classical types and computer checks for exceptional types. For more general complete simplicial fans, 
an affirmative answer to the above question is only known for \emph{cyclic groups} in~\cite{Browder}.

Our main theorem answers Stanley’s question affirmatively for all smooth projective toric varieties and, in this setting, removes the properness assumption.

\begin{theorem}\label{thm:main}
Let \(\Sigma\) be a smooth polytopal $N$-rational fan, and let a finite group
\(G\) act on \(\Sigma\) preserving the lattice \(N\). Let $X_\Sigma$ be the associated smooth projective toric variety. Then there is a finite \(G\)-set \(S\)
such that, after forgetting the cohomological grading,
\[H^*(X_\Sigma;\mathbb{C})\cong\operatorname{Map}(S,\mathbb{C})\]
as a \(G\)-representation. 
\end{theorem}

We remark that the $\mathbb{Q}$-version is also true by a standard Noether--Deuring argument; see for example,~\cite[Theorem~(29.7)]{CurtisReiner}. 

The conclusion is necessarily ungraded. Even for fans of Weyl chambers, the
cohomological grading is not always compatible with a permutation basis
\cite{Stembridge}.

A bit surprisingly, the idea of the proof is based on \emph{mirror symmetry}. 
The quantum cohomology $QH^*(X_{\Sigma})$ provides a flat deformation of ordinary cohomology $H^*(X_\Sigma)$. By choosing a $G$-equivariant direction of K\"ahler parameters, we are able to get a $\mathbb{C}$-family of $G$-algebras, whose zero fiber is $H^*(X_\Sigma)$ and generic fiber is typically semisimple. 
If the semisimplicity were known, the generic fiber would be the coordinate ring of some finite $G$-set, thus a permutation representation (see Proposition \ref{prop:semisimple}). 

However, there are two obstructions to making this argument work in general: 
\begin{itemize}
    \item semisimplicity is sensitive to specialization, a specific choice of $G$-invariant direction need not meet the semisimple locus in the required way; 
    \item in general, even in the Fano case, the (small) quantum cohomology is not generally semisimple;
\end{itemize}
see~\cite{OstroverTyomkin} for explicit examples. 
That is how mirror symmetry comes into the picture. 
Mirror symmetry predicts that the enumerative geometry of a smooth projective variety (quantum cohomology) is encoded by the singularity theory of a mirror Landau--Ginzburg model (Jacobian algebra). Givental's mirror theorem~\cite{Givental} confirms this principle when the toric variety $X_{\Sigma}$ is Fano. But the naive uncorrected Jacobian presentation need not be equal to the small quantum cohomology outside the Fano setting.
In order to prove our main theorem (Theorem \ref{thm:main}) in full generality, we therefore develop a purely combinatorial replacement for the mirror construction of the Landau--Ginzburg potential, and use commutative algebra to establish the required flatness and finiteness properties. With the help of equivariant Morse approximation of Roberts~\cite{Roberts}, we can directly conclude that a large class of local Jacobian algebras in our construction is actually a permutation representation (Theorem \ref{thm:Jacobiperm}) and obtain the desired permutation representation for all smooth projective toric varieties.

Finally, the finite $G$-set constructed in the proof satisfies
\begin{equation*}
  |S^g|=\chi(X_\Sigma^g), \quad \text{ for any $g\in G,$ }
\end{equation*}
where $\chi(X_\Sigma^g)$ is the Euler characteristic of the fixed locus $X_\Sigma^g$. We prove this identity directly and compare it with the specialization at
$q=1$ of Gui's equivariant $h$-polynomial formula in~\cite{Gui}. The fixed-point
formula explains the character of $S$; the deformation and invariant Morse
approximation establish the existence of the actual finite $G$-set.

The paper is organized as follows.  Section~\ref{sec:2} records the required facts about toric varieties and permutation representations. 
Section~\ref{sec:3} completes the proof of
Theorem~\ref{thm:main} and explains the mirror-symmetric meaning of the
construction. Section~\ref{sec:4} studies fixed-point counts and compares them with the
equivariant character formula.

\subsection*{Acknowledgments}
We thank Changzheng Li and John Stembridge for valuable discussions and communications. 
This research began at the PKU Algebraic Combinatorics Experience (PACE 2026)
at the Beijing International Center for Mathematical Research, Peking
University. We thank Yibo Gao for organizing this program and bringing us
together. The first author is supported in part by NSFC: 12471309.

\subsection*{AI disclosure}
ChatGPT was used during the preparation of the manuscript as a technical assistant for checking calculations and arguments, locating relevant references, improving the exposition and identifying mathematical and typographical errors. The mathematical ideas, conceptual framework, and proof strategies were developed by the authors. All AI-assisted content was carefully verified by the authors, who take full responsibility for the manuscript.

\section{Preliminaries}\label{sec:2}

\subsection{Fans and toric varieties}
We fix notation on fans and toric varieties that will be used throughout the paper. A more thorough background on fans and toric varieties can be found in~\cite{CLS,Fulton}.

Fix a complex algebraic torus $T$ of rank $d$. 
Denote its cocharacter and character lattices by 
$$N=X_*(T),\qquad M=X^*(T)=\operatorname{Hom}(N,\mathbb{Z}).$$
Let $\Sigma\subseteq N_{\mathbb{R}}:=N\otimes_\mathbb{Z}\mathbb{R}$ be a smooth (that is, each cone of $\Sigma$ is generated by part of a $\mathbb{Z}$ basis of the lattice $N$) polytopal fan.  We denote its set of
rays by $\Sigma(1)$ and the primitive generator of
$\rho\in\Sigma(1)$ by $v_\rho\in N$.  The associated toric variety is denoted
by $X_\Sigma$. 
Recall that smoothness of $\Sigma$ is equivalent to smoothness
of $X_\Sigma$, and $\Sigma$ is polytopal if and only if $X_\Sigma$ is
projective; see~\cite[Chapters~3 and~6]{CLS} or~\cite{Fulton}.

Let a finite group $G$ act on $N$ by lattice automorphisms preserving
$\Sigma$. It acts contragrediently on $M$ by
\begin{equation*}
  (gm)(u)=m(g^{-1}u),
\end{equation*} and it induces an  algebraic action of $G$ on $X_\Sigma$.  

\subsection{Permutation representation}
Let $G$ be a finite group. 
We say that a complex $G$-representation $V$ is a \emph{permutation representation} if as representations 
\begin{equation*} 
   V\simeq \operatorname{Map}(S,\mathbb{C}) 
\end{equation*}
for a finite $G$-set $S$. 
Note that being a permutation representation is a strong restriction. For example, for any $g\in G$, the character 
$$\chi_V(g)=|S^g|:=\#\{s\in S: gs=s\}\in \mathbb{Z}_{\geq 0}.$$

\begin{lemma}\label{lem:permisind}
Permutation representations are closed under direct sums and inductions. 
\end{lemma}
\begin{proof}
It follows from the following identities
$$
\operatorname{Map}(S_1,\mathbb{C})\oplus 
\operatorname{Map}(S_2,\mathbb{C})\simeq 
\operatorname{Map}(S_1\sqcup S_2,\mathbb{C}),$$
$$\operatorname{Ind}_H^G(\operatorname{Map}(S,\mathbb{C}))
\simeq \operatorname{Map}(G\times_H S,\mathbb{C}),$$
where $G\times_H S =G\times S\big/ \langle (gh,s)\sim (g,hs):h\in H\rangle $. 
\end{proof}

We need the following observation. 

\begin{prop}\label{prop:semisimple}
Assume that $G$ acts by algebra automorphisms on a finite-dimensional commutative $\mathbb{C}$-algebra $R$. 
If $R$ is semisimple, then $R$ is a permutation representation of $G$. 
\end{prop}
\begin{proof}
A finite-dimensional commutative semisimple $\mathbb{C}$-algebra is a product of
copies of $\mathbb{C}$. Thus the spectrum $\operatorname{Spec}R$ is a finite reduced $G$-scheme, or equivalently, a discrete set of points with $G$-action. 
Hence $R\cong \operatorname{Map}(\operatorname{Spec}R,\mathbb{C})$ is a permutation representation.
\end{proof}

We also need the following fact about families of representations. 

\begin{lemma}\label{lem:deformationlem}
Let $B$ be a connected scheme locally of finite type over $\mathbb{C}$, with trivial $G$-action. 
Let $\mathcal{E}$ be a finite rank $G$-equivariant vector bundle on $B$. 
Then all closed fibers of $\mathcal{E}$ are isomorphic as $G$-representations. 
\end{lemma}
\begin{proof}
This is because the isotypic component $\mathcal{E}_\chi$ for any irreducible representation $\chi$ of $G$ is a direct summand of $\mathcal{E}$, hence again a vector bundle on the connected base $B$.
\end{proof}

Let $(\mathfrak{X},p)$ be a smooth complex analytic (resp. algebraic) germ of dimension $n$. For $f\in\mathcal{O}_{\mathfrak X,p}$, we define the \emph{local Jacobian algebra}
\begin{equation} \label{eq:local-jacobian}
  J_p(f) = 
{\mathcal{O}_{\mathfrak{X},p}}\big/
{\left< \frac{\partial f}{\partial z_1},\ldots,\frac{\partial f}{\partial z_n}\right>},
\end{equation}
where $z_1,\ldots,z_n$ are local coordinates. Intrinsically, the ideal in
the denominator is
\[
  \{\xi(f):\xi\in\operatorname{Der}_\mathbb{C}(\mathcal{O}_{\mathfrak X,p})\}.
\]
Thus the definition is coordinate-independent. 
We say that $f$ has an isolated critical point at $p$ if $df(p)=0$ and $\dim_\mathbb{C} J_p(f)<\infty$. 

\begin{theorem}[Local permutation theorem]\label{thm:Jacobiperm}
Assume that a finite group $G$ acts holomorphically (resp., algebraically) on a smooth complex germ $(\mathfrak{X},p)$. 
Suppose that the tangent space $T_p\mathfrak{X}$ is the complexification of a real $G$-representation. If $f\in \mathcal{O}_{\mathfrak{X},p}$ is $G$-invariant and has an isolated critical point at $p$, then 
$J_p(f)$ is a permutation representation of $G$. 
\end{theorem}
\begin{proof}
In the algebraic case, analytification identifies the finite-length local Jacobian algebra with its analytic counterpart, equivariantly for the group action; thus it suffices to prove the analytic assertion.

Choose local coordinates \(\xi:(\mathfrak X,p)\to(T_p\mathfrak X,0)\) with \(d\xi_p=\mathrm{id}\), and replace \(\xi\) by the averaged map over $G$. Then the action is linearized analytically and we can assume $G$-equivariantly $(\mathfrak{X},p)=(T_p\mathfrak{X},0)$. When $T_p\mathfrak{X}$ is a complexification of a real $G$-representation, Roberts's invariant Morse approximation theorem~\cite[Theorem 4.1]{Roberts} states that $f$ admits a $G$-invariant Morse deformation $f_t$ for $t$ defined on a neighbourhood of $0$ such that $f_0=f$ and, for $t\neq 0$, $f_t$ has only non-degenerate critical points. 
Fix a sufficiently small $G$-invariant Milnor neighborhood $B$ and take a sufficiently small generic parameter $t$. Let 
$$\operatorname{Crit}(t) = \{q \in B : d f_t(q) = 0\},$$
which carries an action of $G$.
Montaldi's theorem~\cite[Theorem 4.2 and equation (4.1)]{Montaldi} states that, as representations of $G$, 
\begin{equation}\label{eq:deformJalg}
J_p(f)\simeq \prod_{q\in \operatorname{Crit}(t)}J_q(f_t).
\end{equation}

Since $f_t$ has only non-degenerate critical points, each $J_q(f_t)$ is one-dimensional, thus the right-hand side of \eqref{eq:deformJalg} is semisimple.
By Proposition \ref{prop:semisimple}, it is a permutation representation of $G$. See also~\cite[Theorem 5.1]{Roberts}.
\end{proof}

\begin{remark}
The realness hypothesis cannot be omitted.  Let $G=\mathbb{Z}/3\mathbb{Z}$ act on $\mathbb{C}$ by
$z\mapsto\zeta z$, where $\zeta$ is a primitive cube root of unity.  The
function $f(z)=z^3$ is invariant, but
\[
  J_0(f)=\mathbb{C}\{z\}/\langle z^2\rangle\simeq\mathbb{C}[z]/\langle z^2\rangle
\]
has character $1+\zeta^{-1}\notin \mathbb{Z}_{\geq 0}$ and is not a permutation representation of $G$.
\end{remark}

\section{Proof of Theorem \ref{thm:main}}\label{sec:3}
The projectivity assumption enters through the following elementary
construction.
\begin{lemma}
There exists a full-dimensional $G$-invariant lattice polytope $P\subset M_{\mathbb{R}}:=M\otimes_\mathbb{Z}\mathbb{R}$, containing $0$ in its interior, whose normal fan is $\Sigma$. 
\end{lemma}
\begin{proof}
Since $\Sigma$ is polytopal, there is a full-dimensional rational polytope $P_0\subset M_{\mathbb{R}}$ with normal fan $\Sigma$. 
Choose a rational point in
the interior of $P_0$, translate it to the origin, and take a
sufficiently divisible positive integral dilation. Then we may assume that $P_0$ is a lattice polytope and $0$ is contained in the interior of $P_0$. 
We can take $P$ to be the Minkowski sum of $gP_0$ for all $g\in G$. Since the normal fan of $gP_0$ is
$g\Sigma=\Sigma$, the normal fan of $P$ is again $\Sigma$ by~\cite[Proposition~7.12]{Ziegler}.
\end{proof}

For any $u\in N$, we define a function
$$
\phi(u)=\max\{\langle x,u\rangle:x\in P\}\in \mathbb{Z}_{\geq 0}. $$
The function $\phi$ is $G$-invariant and linear on every cone of $\Sigma$.
Moreover, it is not hard to see from the definition of normal fan that 
\begin{equation*} 
    \phi(u+v)\leq \phi(u)+\phi(v)
\end{equation*}
with equality if and only if $u$ and $v$ lie in a common cone of $\Sigma$. 

Let us introduce the \emph{deformed group ring}
$$R_\Sigma=\bigoplus_{u\in N}\mathbb{C}[t]\cdot x^u,\qquad 
x^u\cdot x^v = t^{\phi(u)+\phi(v)-\phi(u+v)}x^{u+v}. $$
It is not hard to check that $R_\Sigma$ is a commutative, associative algebra with unit $1:=x^0$, free as a $\mathbb{C}[t]$-module.
Using completeness and smoothness of $\Sigma$, it is not hard to see that $R_{\Sigma}$ is finitely generated by elements $x^{v_\rho}$ for $\rho\in \Sigma(1)$. 
% For any $u\in N$, choose a cone $\sigma\in\Sigma$ containing
% $u$ by completeness of $\Sigma$.  Smoothness of $\Sigma$ gives
% \[
%   u=\sum_{\rho\in\sigma(1)}a_\rho v_\rho,
%   \qquad a_\rho\in\mathbb{Z}_{\geq0}.
% \]
% All partial sums lie in $\sigma$, so the equality case of
% \eqref{eq:subadditivity} gives
% \[
%   x^u=\prod_{\rho\in\sigma(1)}(x^{v_\rho})^{a_\rho}.
% \]
% Thus finitely many elements $x^{v_\rho}$
%  for $\rho \in \Sigma(1)$ generate $R_\Sigma$ as a $\mathbb{C}[t]$-algebra.
By the $G$-invariance of $\phi$, the group $G$ acts on $R_\Sigma$ by $\mathbb{C}[t]$-algebra automorphisms defined by $gx^{u}:=x^{gu}$. 

For $m\in M$, let
$$  \theta_m
  =\sum_{\rho\in\Sigma(1)}
   \langle m,v_\rho\rangle x^{v_\rho}\in R_\Sigma.
$$
Let us introduce the ideal
$$  J_\Sigma=\langle\theta_m:m\in M\rangle\trianglelefteq R_\Sigma.
$$
For any basis $m_1,\ldots,m_d$ of $M$, the ideal $J_\Sigma$ is generated by
$\theta_{m_1},\ldots,\theta_{m_d}$.
We fix a basis $m_1,\ldots,m_d$ once and retain it throughout.
Since $g\theta_m=\theta_{gm}$ for any $g\in G,m\in M$, the ideal $J_\Sigma$ is $G$-stable.
Let us consider 
$$A_\Sigma := R_\Sigma/J_\Sigma.$$

\begin{lemma}\label{lem:iso0}
We have $G$-equivariant isomorphisms
$$
R_\Sigma/tR_\Sigma \simeq H^*_T(X_\Sigma; \mathbb{C}),\qquad 
A_\Sigma/tA_\Sigma \simeq H^*(X_\Sigma; \mathbb{C}).   
$$
\end{lemma}
\begin{proof}
Note that $R_{\Sigma}/tR_{\Sigma}$ is isomorphic to the algebra $\mathbb{C}[\Sigma]:=\bigoplus_{u\in N}\mathbb{C}\cdot x^u$ with multiplication
$$
x^u\cdot x^v = \begin{cases}
x^{u+v}, & \text{$u,v$ lie in a common cone of $\Sigma$},\\
0, & \text{otherwise}. 
\end{cases}$$
Since $\Sigma$ is assumed to be complete and smooth, the \emph{Stanley--Reisner ring} of $\Sigma$
$$
\frac{\mathbb{C}[y_\rho: \rho\in \Sigma(1)]}
{\left<y_{\rho_1}\cdots y_{\rho_k}: 
\operatorname{cone}(\rho_1,\ldots,\rho_k)\notin \Sigma\right>}
$$
is isomorphic to $\mathbb{C}[\Sigma]$ by 
$y_{\rho}\mapsto x^{v_\rho}$; compare~\cite[Section~5]{BorisovChenSmith} and~\cite{IchimRomer}. 
In particular, the above presentation agrees with the presentation of $H^*_T(X_\Sigma;\mathbb{C})$, see~\cite[Theorem 12.4.14]{CLS}. Note that the image of $\theta_m$ modulo $t$ gives precisely the \emph{toric linear system of parameters} of $H^*_T(X_\Sigma;\mathbb{C})$, hence we have the second isomorphism, see~\cite[Theorem 12.4.4]{CLS}.
\end{proof}

\begin{remark}\label{rem:Rees}
The algebra $R_\Sigma$ is an elementary Rees-type form of the deformed group
ring appearing in toric mirror constructions; compare
\cite[Section 5]{BorisovChenSmith}.  Smoothness of $\Sigma$ is used to identify the special
fiber $R_\Sigma/tR_\Sigma$ with the ordinary Stanley--Reisner ring rather than a more general
toric face ring.
\end{remark}

Consider the ideal 
$$\mathfrak{m} = \langle t\rangle+\left< x^{v_\rho}:\rho\in \Sigma(1)\right>\trianglelefteq R_\Sigma.$$
By Lemma \ref{lem:iso0}, $\mathfrak{m}$ is the unique prime ideal containing $\langle t\rangle+J_\Sigma$.
In particular, it is a $G$-stable maximal ideal.
% To see this, by \eqref{eq:iso}, the quotient of $R_\Sigma$ by $(t)+J_\Sigma$ is the connected graded Artinian algebra $H^*(X_\Sigma;\mathbb{C})$, whose positive-degree ideal is its unique prime ideal.
We use the same symbol $\mathfrak{m}$ for the image of $\mathfrak{m}$ in $A_\Sigma$.

\begin{lemma}There exists a $G$-invariant element $f\in A_{\Sigma}\setminus\mathfrak{m}$ such that $(A_{\Sigma})_f$ is flat over $\mathbb{C}[t]$. 
\end{lemma}
\begin{proof}
The ring $R_\Sigma$ is free as a $\mathbb{C}[t]$-module, so $t$ is a non-zero-divisor in $R_\Sigma$.
By equivariant formality, $R_\Sigma/tR_\Sigma\simeq H_T^*(X_{\Sigma})$ is free over $\mathbb{C}[\theta_{m_1},\ldots,\theta_{m_d}]\simeq H_T^*(\mathsf{pt})$,
hence the generators 
$t,\theta_{m_1},\ldots,\theta_{m_d}$
of $\langle t\rangle + J_\Sigma$ form a regular sequence in $R_\Sigma$, thus in $(R_{\Sigma})_{\mathfrak{m}}$. 

In a Noetherian local ring, any permutation of regular sequence is still a regular sequence (see~\cite[{\href{https://stacks.math.columbia.edu/tag/00LJ}{Lemma 10.68.4}}]{Stacks}). Moving $t$ to the last position, we see that $t$ is a non-zero-divisor on 
$$(R_{\Sigma})_{\mathfrak{m}}/
\langle \theta_{m_1},\ldots,\theta_{m_d}\rangle 
=(A_{\Sigma})_{\mathfrak{m}}.$$
That is $(A_\Sigma)_{\mathfrak{m}}$ is torsion-free, thus flat over the discrete valuation ring $\mathbb{C}[t]_{\langle t\rangle}$. 
%\cite[{\href{https://stacks.math.columbia.edu/tag/0AUW}{0AUW}}]{Stacks}. 
% Hence $A_\Sigma$ is flat over
% $\mathbb{C}[t]$ at $\mathfrak m$.
Since the flat locus is open~\cite[{\href{https://stacks.math.columbia.edu/tag/0399}{Theorem 37.15.1}}]{Stacks}, we can find $f\notin \mathfrak{m}$ such that $(A_\Sigma)_f$ is a flat $\mathbb{C}[t]$-module. Note that $\mathfrak{m}$ is a $G$-stable maximal ideal. Replacing $f$ by $\prod_{g\in G} gf$, we can further choose $f\notin \mathfrak{m}$ to be $G$-invariant. 
\end{proof}

Now, let us switch to the language of algebraic geometry. 
Let $\mathfrak{X}=\operatorname{Spec}(A_\Sigma)_f$. 
The $\mathbb{C}[t]$-algebra structure of $A_\Sigma$ defines a morphism 
$$\kappa: \mathfrak{X}\longrightarrow \mathbb{A}^1. $$
The special fiber $\kappa^{-1}(0)$ is supported at the single point $\mathfrak{m}$. Indeed, as $f\notin \mathfrak{m}$
the image of $f$ in the local Artinian algebra $A_\Sigma/tA_\Sigma$ is a
unit, so localization does not change the special fiber
\begin{equation}\label{eq:central-fiber}
  \mathbb{C}[\mathfrak X_0]
  =(A_\Sigma)_f/t(A_\Sigma)_f=A_\Sigma/tA_\Sigma.
\end{equation}

By~\cite[{\href{https://stacks.math.columbia.edu/tag/02LP}{Lemma 37.41.6}}]{Stacks}, we can find an elementary \'etale neighbourhood $\varphi:(U,\mathfrak{o})\to (\mathbb{A}^1,0)$ and a decomposition into open-and-closed subschemes
$$\mathfrak{X}_U:= \mathfrak{X}\times_{\mathbb{A}^1}U 
=W \amalg V$$
such that 
$V\to \mathfrak{X}_U\to U$ is 
finite (denoted by $\pi$), and 
 $W\to \mathfrak{X}_U\to U$ has empty fiber at $\mathfrak{o}$. 
After replacing $U$ by the connected component containing $\mathfrak o$, we may
assume that $U$ is connected. 
The situation is summarized by 
$$\xymatrix{
V\ar@{^(->}[r]\ar[dr]_{\pi}&\mathfrak{X}_U\ar[r]\ar[d] & \mathfrak{X}\ar[d]^\kappa\\
& U\ar[r]^{\varphi}& \mathbb{A}^1
& 
}$$
Since an \'etale morphism is flat, and our construction of $f$ implies that $\kappa$ is flat, all morphisms in the diagram are flat. 
As a result, the sheaf $\pi_*\mathcal{O}_V$ is a vector bundle over $U$.  

As $V$ is finite and flat over $U$ and $U$ is connected, the decomposition $\mathfrak{X}_U=W\amalg V$ is unique. 
% If there is another choice of $V$, then the difference (if nonempty) is flat and free over $U$ but has empty fiber at $\mathfrak{o}$. This is impossible as $U$ is connected. 
In particular, $V$ is $G$-stable. As a result, $\pi_*\mathcal{O}_V$ is a $G$-equivariant vector bundle over $U$, where $G$ acts trivially on the base $U$.

For any closed point $u\in U$, let $V_{u}=\pi^{-1}(u)\subset \kappa^{-1}(\varphi(u))=\mathfrak{X}_{\varphi(u)}$ be the scheme-theoretic fiber. 
The fiber of $\pi_*\mathcal{O}_V$ at $u$ is the coordinate ring
$\mathbb{C}[V_u]$. 
Since $\pi_*\mathcal{O}_V$ is a $G$-equivariant vector bundle, by Lemma \ref{lem:deformationlem}, we have 
\begin{equation}\label{eq:iso1}
\mathbb{C}[V_{u}]\simeq \mathbb{C}[V_{\mathfrak{o}}]=\mathbb{C}[\mathfrak{X}_{0}]
\end{equation}
as $G$-representations.

Choose a closed point $u_0\in U$ with $\varphi(u_0)=t_0\neq 0$. 
Such a point exists because $\varphi$ is \'etale at $\mathfrak{o}$. 
Put 
$$\overline{R}_{\Sigma} := R_\Sigma/(t-t_0)R_\Sigma.$$
Define the ambient fiber
$$X_{t_0}=\operatorname{Spec}\bigl(A_\Sigma/(t-t_0)A_\Sigma\bigr)
=\operatorname{Spec}(
\overline{R}_\Sigma/\langle\theta_{m_1},\ldots,\theta_{m_d}\rangle).
$$
The fiber $V_{u_0}$ is open and closed in $\mathfrak{X}_{t_0}=X_{t_0}\cap D(f)$. 
Hence, at each point of $V_{u_0}$, its local ring agrees with the
corresponding local ring of $X_{t_0}$.
Since $V_{u_0}$ is closed and open in $\mathfrak{X}_{t_0}$, we have 
$$
\mathbb{C}[V_{u_0}]=\prod_{\mathfrak{p}\in (V_{u_0})_{\rm red}}
\mathbb{C}[V_{u_0}]_{\mathfrak{p}}
\cong \bigoplus_{\mathfrak{p}\in (V_{u_0})_{\rm red}}
\mathbb{C}[\mathfrak{X}_{t_0}]_{\mathfrak{p}}
= \bigoplus_{\mathfrak{p}\in (V_{u_0})_{\rm red}}
\mathbb{C}[X_{t_0}]_{\mathfrak{p}},$$
the isomorphism holds because $(V_{u_0})_{\rm red}$ is a finite scheme.
Let us decompose $(V_{u_0})_{\rm red}
\subset (X_{t_0})_{\rm red}$ into $G$-orbits
$$(V_{u_0})_{\rm red}=G\mathfrak{p}_1\sqcup G\mathfrak{p}_2\sqcup \cdots\sqcup G\mathfrak{p}_m,$$
where $\mathfrak{p}_1,\ldots,\mathfrak p_m$ are the representatives for the
$G$-orbits in $(V_{u_0})_{\mathrm{red}}$.
Let $H_i\leq G$ be the stabilizer of $\mathfrak{p}_i$. 
Then as $G$-representations, we have
\begin{equation*}
\mathbb{C}[V_{u_0}]
\simeq \bigoplus_{i=1}^m
\operatorname{Ind}_{H_i}^G \big(\mathbb{C}[X_{t_0}]_{\mathfrak{p}_i}\big).
\end{equation*}

\begin{lemma}\label{lem:iso3}
The ring $\mathbb{C}[X_{t_0}]_{\mathfrak{p}_i}$ is a permutation representation of $H_i$. 
\end{lemma}
\begin{proof}
Note that we have a $G$-equivariant isomorphism between $\overline{R}_\Sigma$ and the group ring
$$\psi:\overline{R}_\Sigma \stackrel{\sim}\longrightarrow
\mathbb{C}[N],\qquad 
x^u \longmapsto t_0^{\phi(u)}e^u.$$
Since 
$\operatorname{Spec}(\mathbb{C}[N])=\mathbb{C}^\times\otimes_{\mathbb{Z}}M=T^\vee$, we can identify the points $\mathfrak{p}_i\in T^\vee$. 
Let us consider the $G$-invariant Laurent polynomial (called the potential) 
\begin{equation}\label{eq:potential}
\mathcal{W} = \sum_{\rho\in \Sigma(1)} t_0^{\phi(v_\rho)}e^{v_\rho}\in \mathbb{C}[N]^G. 
\end{equation}
Let us pick a basis $m_1^*,\ldots,m_d^*$ of $N$ dual to $m_1,\ldots,m_d$ and put $z_i=e^{m_i^*}$. A direct calculation gives 
$$z_i 
\frac{\partial \mathcal{W}}{\partial z_i}
= \sum_{\rho\in \Sigma(1)}
\langle m_i,v_\rho\rangle t_0^{\phi(v_\rho)} e^{v_\rho}=\psi(\theta_{m_i}), \qquad i=1,\dots,d.$$
Since $z_i$ is invertible on $T^\vee$, we have 
$$\mathbb{C}[X_{t_0}]_{\mathfrak{p}_i}
\cong J_{\mathfrak{p}_i}(\mathcal{W}),$$
the local Jacobian algebra of $\mathcal{W}$.
Here the right-hand side is actually the analytic local Jacobian algebra from
\eqref{eq:local-jacobian}. 
Indeed, because \(V_{u_0}\) is finite and its local ring at \(\mathfrak p_i\) agrees with \(\mathcal O_{X_{t_0},\mathfrak p_i}\), the latter is a finite-dimensional \(\mathbb C\)-algebra. The algebraic-to-analytic local homomorphism induces an isomorphism on this finite-length quotient; see~\cite[Propositions~3, 28, and~29]{SerreGAGA}. 
Finiteness of the scheme $V_{u_0}$ implies that every such
$\mathfrak{p}_i$ is an isolated critical point of $\mathcal{W}$.

Translation by \(\mathfrak p_i^{-1}\) is \(H_i\)-equivariant because, for \(h\in H_i\),
\[
h(\mathfrak p_i^{-1}z)=(h\mathfrak p_i)^{-1}h(z)=\mathfrak p_i^{-1}h(z).
\]
Hence we can identify $H_i$-modules \(T_{\mathfrak p_i}T^\vee\cong T_1T^\vee\cong M\otimes_\mathbb Z\mathbb C\), the complexification of \(M_\mathbb R\).
In particular, it is the complexification of a real $H_i$-representation. 
By Theorem \ref{thm:Jacobiperm}, 
$J_{\mathfrak{p}_i}(\mathcal{W})$ is a permutation representation of $H_i$. 
\end{proof}

Combining Lemma \ref{lem:permisind} and Lemma \ref{lem:iso3}, we see that $\mathbb{C}[V_{u_0}]$ is a permutation representation of $G$. 
By Lemma \ref{lem:iso0}, \eqref{eq:central-fiber} and \eqref{eq:iso1}, 
we can conclude that $H^*(X_\Sigma; \mathbb{C})$ is a permutation representation of $G$. 

\begin{remark}
The support function $\phi$ determines a $G$-invariant ample divisor class
$D_{-\phi}$ on $X_\Sigma$. 
It therefore determines a $\mathbb{Z}$-linear degree map
\[
\deg_{\phi}:H_{2}(X_{\Sigma};\mathbb{Z})\longrightarrow \mathbb{Z},\qquad \deg_{\phi}(\beta) = D_{-\phi}\cdot \beta,
\]
which is strictly positive on \(\operatorname{Eff}(X_{\Sigma})\setminus \{0\}\).
%where $\operatorname{Eff}(X_\Sigma)\subset H_2(X_\Sigma;\mathbb{Z})$ is the semigroup generated by effective curves.

If \(X_{\Sigma}\) is Fano, Givental's mirror theorem~\cite{Givental} identifies this ring \(A_{\Sigma}\) with the specialized small quantum cohomology of \(X_{\Sigma}\) under \(q^{\beta}\to t^{\deg_{\phi}(\beta)}\). The Laurent polynomial \eqref{eq:potential} is then the associated Landau--Ginzburg potential, and the quotient by its logarithmic derivatives is its Jacobian algebra; see also~\cite[Proposition 3.3]{OstroverTyomkin}.

In the Fano case, quantum cohomology is finite free over the parameter ring,
so the whole one-parameter family can be used.  For a general smooth
projective toric variety, the algebra $A_\Sigma$ need not furnish a globally finite flat family with this direct quantum-cohomological interpretation.
For example, for the Hirzebruch surfaces 
$$X_\Sigma = \mathbb{F}_a=\mathbb{P}(\mathcal{O}_{\mathbb{P}^1}\oplus \mathcal{O}_{\mathbb{P}^1}(a)), \quad a>0$$
with the standard fan convention, the special fiber $A_\Sigma/tA_\Sigma \simeq H^*(X_\Sigma)$ is $4$-dimensional, and an easy computation shows that $A_\Sigma/(t-t_0)A_{\Sigma}$ for generic $t_0\neq 0$ is of  dimension $\max(4,a+2)$. 

Another example is 
$$X_\Sigma = \text{blowup of  $\mathbb{P}^1\times \mathbb{P}^1$ at $\{0,\infty\}\times \{0,\infty\}$}.$$
The fan $\Sigma$ is the Coxeter fan of type $B_2$ and we consider the full Weyl group $B_2$ action on it. 
For some specific values $t_0\neq 0$ (depending on the choice of $\phi$), $A_\Sigma/(t-t_0)A_\Sigma$ is infinite dimensional, and for other $t_0$'s, $A_\Sigma/(t-t_0)A_\Sigma$ is $8$-dimensional. 

The \'etale construction above isolates the finite flat branch that
specializes to ordinary cohomology and is sufficient for the representation-theoretic argument, avoiding any global quantum-cohomological identification.

Figure~\ref{fig:finite-branch} illustrates the local geometry near $t=0$:
the component $V$ is the finite flat branch containing the special fiber,
whereas $W$ contains the remaining branches after \'etale base
change.
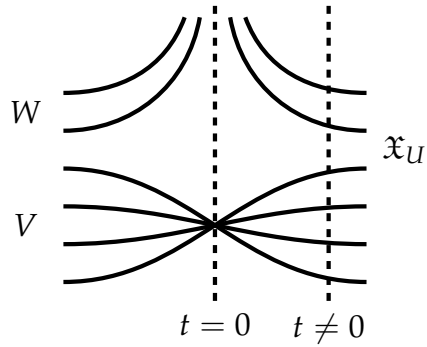
\begin{figure}[ht]
\centering
\begin{tikzpicture}[scale=1]
  \node at (2.5,1) {$\mathfrak X_U$};
  \node at (-2.5,0) {$V$};
  \node at (-2.5,1.5) {$W$};
  \draw[ultra thick] (-2,1.25) to[out=0,in=-100] (-0.2,2.75);
  \draw[ultra thick] (2,1.25) to[out=180,in=-80] (0.2,2.75);
  \draw[ultra thick] (-2,1.75) to[out=0,in=-110] (-0.4,2.75);
  \draw[ultra thick] (2,1.75) to[out=180,in=-70] (0.4,2.75);
  \draw[ultra thick] (-2,0.75) to[out=0,in=180] (2,-0.75);
  \draw[ultra thick] (-2,0.25) to[out=0,in=180] (2,-0.25);
  \draw[ultra thick] (-2,-.25) to[out=0,in=180] (2,0.25);
  \draw[ultra thick] (-2,-0.75) to[out=0,in=180] (2,0.75);
  \draw[ultra thick,dashed] (0,-1) node[below] {$t=0$} to (0,3);
  \draw[ultra thick,dashed] (1.5,-1) node[below] {$t\ne0$} to (1.5,3);
\end{tikzpicture}
\caption{The finite flat branch}
\label{fig:finite-branch}
\end{figure}
\end{remark}

\section{Euler characteristic interpretation of the characters} \label{sec:4}

In this section, we are interested in the geometric interpretation of the character of $g$ on $H^{*}(X_\Sigma)$
$$
\operatorname{Tr}\big(g\mid H^{*}(X_\Sigma)\big)
:=\sum_{i\geq 0}\operatorname{Tr}(g\mid H^{i}(X_\Sigma)).$$

\begin{prop}
Let $S$ be any finite $G$-set such that \(H^*(X_\Sigma;\mathbb{C})\cong\operatorname{Map}(S,\mathbb{C})\)
as \(G\)-representations, whose existence follows from Theorem \ref{thm:main}. Then we have 
$$
\operatorname{Tr}\big(g\mid H^{*}(X_\Sigma)\big)
=|S^g|.$$
\end{prop}

\begin{prop}\label{prop:char=Eulercharoffixedloci}
For any $g\in G$, we have 
$$
\operatorname{Tr}\big(g\mid H^{*}(X_\Sigma)\big)
=\chi(X_\Sigma^g)$$
where $\chi(X_\Sigma^g)$  is the Euler characteristic of the fixed locus $X_\Sigma^g$. 
\end{prop}
\begin{proof}
Since the odd cohomology vanishes,
$$
\operatorname{Tr}\big(g\mid H^*(X_\Sigma)\big)
  =\sum_{i\geq 0}(-1)^i
   \operatorname{Tr}\big(g\mid H^i(X_\Sigma)\big).
$$
Choose a finite $\langle g\rangle$-equivariant triangulation $K$ of the
compact smooth manifold $X_\Sigma$ and replace it by its barycentric
subdivision; see~\cite{Illman}.  Every simplex stabilized setwise by $g$
is then fixed pointwise, because its vertices are barycenters of a strictly
increasing chain of faces of distinct dimensions.  The Hopf trace theorem
\cite[p.~122]{Munkres}, applied to the simplicial cochain complex, gives
\begin{align*}
\sum_{i\geq 0}(-1)^i
 \operatorname{Tr}\bigl(g\mid H^i(X_\Sigma)\bigr)
&=\sum_{i\geq 0}(-1)^i
 \operatorname{Tr}\bigl(g\mid C^i(K)\bigr)\notag\\
&=\sum_{i\geq 0}(-1)^i f_i(K^{g})
 =\chi(X_\Sigma^g),
\end{align*}
where $f_i(K^{g})$ is the number of $i$-simplices of the fixed subcomplex $K^{g}$.
\end{proof}

\begin{remark}
We remark that the proof here can be viewed as a version of Atiyah--Bott localization theorem. But we do not find a proper reference for non-isolated fixed points. 
% One can also see this as follows. 
% Let $G$ be a reductive group acting on a smooth projective variety $X$. 
% Then for any $g\in G$, 
% $$\sum_{i\geq 0}(-1)^i\operatorname{Tr}(g|H^{i}(X)) = \chi(X^g). $$
% This can be seen as follows. Note that $X^g=\Delta_X\cap \Gamma_g$, where 
% $\Delta_X$ (resp. $\Gamma_g$) is the diagonal (resp. graph of $g$) in $X\times X$. Pick a homogeneous basis $\{\alpha_j\}$ of $H^*(X)$, we have
% $$[\Delta_X]=\sum_j \alpha_j\otimes \alpha_j^*=\sum_j \alpha_j^*\otimes \alpha_j,\qquad [\Gamma_{g}] = \sum_j \alpha_j\otimes g\alpha_j^*.$$
% Then 
% $$\int_{X\times X} [\Delta_X]\cdot [\Gamma_g]
% = \sum_{j}(-1)^{|a_j|} \alpha_j\otimes g\alpha_j^*
% =\sum_{i\geq 0}(-1)^i\operatorname{Tr}(g\mid H^{i}(X)).$$
% On the other hand, since $G$ is reductive, the fixed loci $X^g$ is smooth, thus we can apply excess intersection formula to conclude 
% $$[\Delta_X]\cdot [\Gamma_g]
% = \sum_{Z\subseteq X^g} [Z]\cdot e(E_Z)
% $$
% where $Z$ runs over all connected component of $X^g$, and $e(E_Z)$ is the Euler class of the excess bundle. 
% In this case, it is just the tangent bundle of $Z$. 
% In all, we get 
% $$\int_{X\times X} [\Delta_X]\cdot [\Gamma_g]=\sum_{Z\subseteq X^g}
% \int_Z e(T_Z)=\sum_{Z\subseteq X^q}\chi(Z)=\chi(X^g). $$
\end{remark}

For $\sigma\in \Sigma$ with $g\sigma=\sigma$, let
\begin{equation*}
\psi_\sigma(g):N_\mathbb{R}/\operatorname{span}_\mathbb{R}(\sigma)
  \longrightarrow N_\mathbb{R}/\operatorname{span}_\mathbb{R}(\sigma)
\end{equation*}
be the induced quotient action. 
Let $O(\sigma)\subseteq X_\Sigma$ be the orbit corresponding to
$\sigma\in\Sigma$, and let $\chi_c$ denote the Euler characteristic with compact
support, see~\cite[Section 4.5]{Fulton}.

\begin{lemma}\label{lem:orbit-fixed}
For every $\sigma\in\Sigma$ such that $g\sigma=\sigma$,
\begin{equation*}
  \chi_c\bigl(O(\sigma)^g\bigr)
  =\det\bigl(I-\psi_\sigma(g)\bigr).
\end{equation*}
No properness assumption on the action is needed.
\end{lemma}

\begin{proof}
Put
\[
N_\sigma=N\cap\operatorname{span}_\mathbb{R}(\sigma),
  \qquad L_\sigma=N/N_\sigma.
\]
The subgroup $N_\sigma$ is saturated, and $O(\sigma)$ is the torus with
cocharacter lattice $L_\sigma$.  The element $g$ induces a lattice
automorphism $A$ of $L_\sigma$, whose realification is
$\psi_\sigma(g)$.  The fixed locus $O(\sigma)^g$ is the kernel of the torus
homomorphism induced by $A-I$.

Apply the Smith normal form to $I-A$.  If its rank is $s$ and its nonzero
invariant factors are $d_1,\ldots,d_s$, then, non-canonically,
\begin{equation*}
  O(\sigma)^g
  \simeq
  \mu_{d_1}\times\cdots\times\mu_{d_s}
  \times(\mathbb{C}^\times)^{\operatorname{rank}(L_\sigma)-s}.
\end{equation*}
If $I-A$ is singular, the positive-dimensional torus factor makes
$\chi_c$ vanish.  If $I-A$ is nonsingular, the fixed locus has
$|\det(I-A)|$ points.

Since $g$ has finite order, every real eigenvalue of $A$ is $1$ or $-1$,
and the non-real eigenvalues occur in conjugate pairs on the unit circle.  A
$1$-eigenvalue makes the determinant zero, a $-1$-eigenvalue contributes
$2$, and a pair $\lambda,\overline\lambda$ contributes
\[
  (1-\lambda)(1-\overline\lambda)=|1-\lambda|^2>0.
\]
Thus $\det(I-A)\geq0$, and the absolute value may be removed.
\end{proof}

\begin{prop} \label{prop:Euler from Det}
For any $g\in G$, we have 
$$\chi(X_\Sigma^g)=\sum_{{g\sigma=\sigma}}
   \det\bigl(I-\psi_\sigma(g)\bigr).$$
\end{prop}
\begin{proof}
Since $gO(\sigma)=O(g\sigma)$, the fixed locus decomposes as
\[
  X_\Sigma^g
  =\coprod_{g\sigma=\sigma}
   O(\sigma)^g.
\]
Additivity of compactly supported Euler characteristic (\cite[Section 4.5]{Fulton}) and
Lemma~\ref{lem:orbit-fixed} give
$$
  \chi_c(X_\Sigma^g)
  =\sum_{g\sigma=\sigma}
   \det\bigl(I-\psi_\sigma(g)\bigr).
$$
The fixed locus $X_\Sigma^g$ is closed in the projective variety $X_\Sigma$ hence is projective, so the ordinary and compactly supported Euler
characteristics agree. 
\end{proof}

\begin{remark}
In~\cite[Theorem~1.1 and equation~(3.6)]{Gui}, the first author proved the following graded version of the above formula
$$
  \mathcal{P}_{\Sigma}(g;q)
  :=\sum_{i\geq0}\operatorname{Tr}(g\mid H^{2i}(X_\Sigma))\cdot q^i
  =\sum_{g\sigma=\sigma}
   \det\bigl(qI-\psi_\sigma(g)\bigr).
$$
Again, no properness assumption is needed. 
After specializing $q=1$, we can also obtain the Euler characteristic interpretation of the character via Proposition \ref{prop:Euler from Det}. 
\end{remark}

In summary, we have the following. 

\begin{theorem}For any $g\in G$, we have
\begin{equation}\label{eq:final-fixed}
\operatorname{Tr}\big(g\mid H^{*}(X_\Sigma)\big)
= |S^g|=\chi(X_\Sigma^g)
=\sum_{g\sigma=\sigma}
   \det\bigl(I-\psi_\sigma(g)\bigr),
\end{equation}
where $S$ is any finite $G$-set such that \(H^*(X_\Sigma;\mathbb{C})\cong\operatorname{Map}(S,\mathbb{C})\)
as \(G\)-representations.
\end{theorem}

\begin{remark}\label{rem:Burnside}
Equation \eqref{eq:final-fixed} is an identity of characters. The character identity by itself does not produce
an actual finite $G$-set. This distinction is essential for non-cyclic groups: element-wise fixed-point counts detect only cyclic subgroups and do not in general establish effectivity in the Burnside ring. The support-function degeneration and invariant Morse approximation are what establish the existence of the effective finite $G$-set $S$.
\end{remark}

\end{document}